\documentclass[a4paper]{amsart}
\usepackage{amscd}
\usepackage{amsmath}
\usepackage{amssymb}
\usepackage{amsthm}
\usepackage{bbm}
\usepackage{stmaryrd}

\usepackage[T1]{fontenc}

\newif\ifpdf
\ifx\pdfoutput\undefined
   \pdffalse        
\else
   \pdfoutput=1     
   \pdftrue
\fi

\ifpdf
   \usepackage[pdftex]{graphicx}
\else
   \usepackage{graphicx}
\fi

\numberwithin{equation}{section} \swapnumbers

\newtheorem{satz}{Satz}[section]

\newtheorem{theorem}[satz]{Theorem}
\newtheorem{proposition}[satz]{Proposition}

\newtheorem{definition}[satz]{Definition}

\newtheorem{remark}[satz]{Remark}

\newtheorem{example}[satz]{Example}

\newcommand{\bbr}{\mathbb{R}}

\newcommand{\bbn}{\mathbb{N}}

\newcommand{\calb}{\mathcal{B}}
\newcommand{\cald}{\mathcal{D}}

\newcommand{\call}{\mathcal{L}}
\newcommand{\calm}{\mathcal{M}}
\newcommand{\caln}{\mathcal{N}}

\newcommand{\fl}{{\rm fl}\,}

\begin{document}

\title[Flatness of invariant manifolds for SPDEs driven by L\'{e}vy processes]{Flatness of invariant manifolds for stochastic partial differential equations driven by L\'{e}vy processes}
\author{Stefan Tappe}
\address{Leibniz Universit\"{a}t Hannover, Institut f\"{u}r Mathematische Stochastik, Welfengarten 1, 30167 Hannover, Germany}
\email{tappe@stochastik.uni-hannover.de}
\begin{abstract}
The purpose of this note is to prove that the flatness of an invariant manifold for a semilinear stochastic partial differential equation driven by L\'{e}vy processes is at least equal to the number of driving sources with small jumps. We illustrate our findings by means of an example.
\end{abstract}
\maketitle

\section{Introduction}

The purpose of this note is to show that an invariant manifold for a semilinear stochastic partial differential equation (SPDE)
\begin{align}\label{SPDE}
\left\{
\begin{array}{rcl}
dr_t & = & (A r_t + \alpha(r_t)) dt + \sigma(r_t) dW_t + \gamma(r_{t-}) dX_t \medskip
\\ r_0 & = & h_0
\end{array}
\right.
\end{align}
in the spirit of \cite{P-Z-book} driven by L\'{e}vy processes with small jumps necessarily has a certain amount of flatness, that is, of linear structure.

A result which is related to the findings of our paper has been provided in \cite{Filipovic-Teichmann} for the particular case of Wiener process driven Heath-Jarrow-Morton (HJM, see \cite{HJM}) interest rate term structure models, namely that under suitable conditions an invariant manifold for the HJM equation necessarily is a foliation, that is, a collection of affine spaces.

In this paper, we deal with general SPDEs of the type (\ref{SPDE}) driven by L\'{e}vy processes, and the intuitive statement of our main results (Theorems~\ref{thm-local} and \ref{thm-global}) is that the flatness of an invariant manifold is at least equal to the number of driving sources with small jumps. 

In order to acquaint the reader with the ideas behind these results, let us present the key concepts and ideas of the proof in an informal way. Denoting by $H$ the state space of the SPDE (\ref{SPDE}), which we assume to be a separable Hilbert space, and by $\calm$ be a finite dimensional submanifold of $H$, we have the following concepts, which are explained in more detail in Section~\ref{sec-flatness} and Appendix~\ref{app-manifolds}:
\begin{itemize}
\item We call $\calm$ invariant for the SPDE (\ref{SPDE}) if for each starting point $h_0 \in \calm$ the mild solution to (\ref{SPDE}) with $r_0 = h_0$ stays on the manifold.

\item For a point $h_0 \in \calm$ the flatness of $\calm$ at $h_0$ is the largest integer $d$ such that some $d$-dimensional subspace $\call \subset H$ is contained simultaneously in all tangent spaces of the manifold $\calm$ locally around $h_0$.

\item Then, the flatness of $\calm$ is defined as the minimum over all these local quantities.
\end{itemize}
As already indicated, throughout this paper we will assume that $\calm$ is an invariant manifold. The volatility $\gamma = (\gamma^k)_{k \in K}$, where $K = \{ 1,\ldots,q \}$ with $q$ denoting the dimension of the L\'{e}vy process $X$, consists of mappings $\gamma^k : H \to H$, $k \in K$. In order to exemplify the ideas behind our result, we assume (for the sake of simplicity) that for each $k \in K$ the L\'{e}vy process $X^k$ makes arbitrary small positive jumps. Then, for each $h \in \calm$ the flatness of $\calm$ at $h$ is of the stated size, and the proof is divided into two steps:
\begin{itemize}
\item For an arbitrary $k \in K$ the volatility $\gamma^k(h)$ belongs to the tangent space to $\calm$ at $h$. Indeed, since the manifold $\calm$ is invariant, it captures every possible jump of $X^k$. Since, in addition, the L\'{e}vy process $X^k$ makes arbitrary small positive jumps, this means that for some $\epsilon > 0$ we have
\begin{align*}
h + x_k \gamma^k(h) \in \calm \quad \text{for all $x_k \in [0,\epsilon]$.}
\end{align*}
In other words, the line segment $\{ h + x_k \gamma^k(h) : x_k \in [0,\epsilon] \}$ is contained in the manifold $\calm$. From an intuitive point of view, it is clear that this implies that $\gamma^k(h)$ belongs to the tangent space to $\calm$ at $h$. We refer to Proposition~\ref{prop-jumps} for the precise formulation of this statement and its proof.

\item Due to the previous step, the linear space $\call$ generated by all $\gamma^k(h)$, $k \in K$ is contained in the tangent space to $\calm$ at $h$, which provides the desired result concerning the flatness of the manifold.
\end{itemize}
The remainder of this note is organized as follows. In Section \ref{sec-flatness} we provide the general framework and present our main results. In Section \ref{sec-example} we illustrate our main results by means of an example; namely we apply our results to the Hull-White extension of the Vasi\u{c}ek model from interest rate theory. For convenience of the reader, in Appendix \ref{app-manifolds} we provide the crucial definitions and results regarding submanifolds in Hilbert spaces.

\section{Flatness of invariant manifolds}\label{sec-flatness}

In this section, we present our main results concerning the flatness of invariant manifolds for SPDEs driven by L\'{e}vy processes.

Let $(\Omega,\mathcal{F},(\mathcal{F}_t)_{t \in \bbr_+},\mathbb{P})$ be a filtered probability space with right-continuous filtration. Let $W$ be a $p$-dimensional Wiener standard process for some $p \in \bbn_0$, and let $X$ be an $q$-dimensional L\'{e}vy process for some $q \in \bbn$, which we assume to be a purely discontinuous martingale with canonical representation $X = x * (\mu^X - \nu)$ in the sense of \cite[Cor. II.2.38]{Jacod-Shiryaev}. Here $\mu^X$ denotes the random measure associated to the jumps of $X$, which is a homogeneous Poisson random measure, and $\nu$ denotes its compensator, which is given by $\nu(dt,dx) = dt \otimes F(dx)$ with $F$ denoting the L\'{e}vy measure of $X$. We assume that $X^1,\ldots,X^q$ are independent, which implies that the L\'{e}vy measure $F$ is given by
\begin{align}\label{F-coord}
F(B) = \sum_{k=1}^q \int_{\bbr} \mathbbm{1}_B(x e_k) F^k(dx), \quad B \in \calb(\bbr^q)
\end{align}
with $e_1,\ldots,e_q$ denoting the unit vectors in $\bbr^q$, and with $F^k$ denoting the L\'{e}vy measure of $X^k$ for $k=1,\ldots,q$. We assume that
\begin{align}\label{F-square-int}
\int_{\bbr} \big( | x |^2 \vee | x |^4 \big) F^k(dx) < \infty \quad \text{for all $k = 1,\ldots,q$.}
\end{align}
The following definition identifies the set of all indices such that the corresponding L\'{e}vy process makes ``small jumps''.

\begin{definition}\label{def-index-set}
We denote by $K$ be the set of all indices $k \in \{ 1,\ldots,q \}$ such that for some $\epsilon > 0$ we have $[0,\epsilon] \subset {\rm supp}(F^k)$ or $[-\epsilon,0] \subset {\rm supp}(F^k)$.
\end{definition}

Let $H$ be a separable Hilbert space and let $A : \cald(A) \subset H \to H$ be the infinitesimal generator of a $C_0$-semigroup $(S_t)_{t \geq 0}$ on $H$. Furthermore, let $\alpha : H \to H$, $\sigma : H \to H^p$ and $\gamma : H \to H^q$ be Lipschitz continuous mappings such that $\sigma^j \in C^1(H)$ for all $j=1,\ldots,p$. We suppose that the semigroup $(S_t)_{t \geq 0}$ is pseudo-contractive, that is
\begin{align*}
\| S_t \| \leq e^{\beta t}, \quad t \geq 0
\end{align*}
for some constant $\beta \in \bbr$. Then, for each $h_0 \in H$ there exists a unique mild solution to the SPDE (\ref{SPDE}), that is, an adapted c\`{a}dl\`{a}g process $r = r^{(h_0)}$ such that
\begin{align*}
r_t &= S_t h_0 + \int_0^t S_{t-s} \alpha(r_s) ds + \sum_{j=1}^p \int_0^t S_{t-s} \sigma^j(r_s) dW_s^j
\\ &\quad + \sum_{k=1}^q \int_0^t S_{t-s} \gamma^k(r_{s-}) dX_s^k, \quad t \in \bbr_+,
\end{align*}
see, for example, \cite{P-Z-book}, \cite{Marinelli-Prevot-Roeckner} or \cite{SPDE}. For what follows, let $\calm$ be a finite dimensional $C^3$-submanifold of $H$, which we assume to be closed as a subset of $H$. We refer to Appendix \ref{app-manifolds} for details about submanifolds in Hilbert spaces.

\begin{definition}
The submanifold $\calm$ is called \emph{invariant} for (\ref{SPDE}) if for all $h_0 \in \calm$ we have $r \in \calm$ up to an evanescent set\footnote[1]{A random set $A \subset \Omega \times \mathbb{R}_+$ is called \emph{evanescent} if the set $\{ \omega \in \Omega : (\omega,t) \in A \text{ for some } t \in \mathbb{R}_+ \}$ is a $\mathbb{P}$-nullset, cf. \cite[1.1.10]{Jacod-Shiryaev}.}, where $r = r^{(h_0)}$ denotes the mild solution to (\ref{SPDE}) with $r_0 = h_0$.
\end{definition}

\begin{remark}
As our first step in order to analyze the flatness of invariant manifolds, we will write the SPDE (\ref{SPDE}) as the SPDE (\ref{SPDE-new}) below, and apply \cite[Thm.~2.8]{Manifolds}. Let us emphasize those of our previous assumptions, which we have exclusively made for an application of this result:
\begin{itemize}
\item We assume the integrability condition (\ref{F-square-int}), which ensures that condition (2.5) from \cite{Manifolds} holds true.

\item We assume that $\calm$ is a $C^3$-submanifold of $H$, and that it is closed as a subset of $H$. This assumption is also required for the mentioned result from~\cite{Manifolds}.
\end{itemize}
\end{remark}

In the sequel, we also assume that the index set $K$, which identifies all L\'{e}vy processes with ``small jumps'', is nonempty. Otherwise, no statement concerning the flatness of $\calm$ is possible, as the following counterexample shows:

\begin{example}
We consider the SPDE
\begin{align}\label{SPDE-N}
\left\{
\begin{array}{rcl}
dr_t & = & \gamma(r_{t-}) dN_t \medskip
\\ r_0 & = & h_0
\end{array}
\right.
\end{align}
on the state space $H = \bbr^2$, which -- after rewriting -- is of the form (\ref{SPDE}). Here $N$ is a Poisson process, and the volatility $\gamma : \bbr^2 \to \bbr^2$ is given by $\gamma(h) = (1,0)$ for all $h \in \bbr^2$. Then the one-dimensional submanifold
\begin{align*}
\calm = \{ (\xi,\sin(2 \pi \xi)) : \xi \in \bbr \}
\end{align*}
is invariant for (\ref{SPDE-N}), which follows from \cite[Thm.~2.11]{Manifolds}, but we have $\fl \calm(h_0) = 0$ for all $h_0 \in \calm$, showing that the flatness of $\calm$ is zero.
\end{example}

The following result shows that in case of invariance all volatilities associated to L\'{e}vy processes with ``small jumps'' are tangential to the submanifold.

\begin{proposition}\label{prop-jumps}
Suppose that the submanifold $\calm$ is invariant for (\ref{SPDE}). Then we have
\begin{align*}
\gamma^k(h) \in T_h \calm \quad \text{for all $k \in K$ and all $h \in \calm$.}
\end{align*}
\end{proposition}

\begin{proof}
We can write the SPDE (\ref{SPDE}) as
\begin{align}\label{SPDE-new}
\left\{
\begin{array}{rcl}
dr_t & = & (Ar_t + \alpha(r_t))dt + \sigma(r_t) dW_t
+ \int_{\bbr^q} \delta(r_{t-},x) (\mu^X(dt,dx) - F(dx) dt)
\medskip
\\ r_0 & = & h_0,
\end{array}
\right.
\end{align}
where $\delta : H \times \bbr^q \to H$ is given by
\begin{align*}
\delta(h,x) = \sum_{k=1}^q x_k \gamma^k(h), \quad (h,x) \in H \times \bbr^q.
\end{align*}
In view of (\ref{F-square-int}), all assumptions of \cite[Thm. 2.8]{Manifolds} are fulfilled, and together with (\ref{F-coord}), for each $k = 1,\ldots,q$ we obtain
\begin{align}\label{jumps-tangential}
h + x_k \gamma^k(h) \in \calm \quad \text{for all $h \in \calm$ and all $x_k \in {\rm supp}(F^k)$.}
\end{align}
Now, let $k \in K$ and $h_0 \in \calm$ be arbitrary, and let $\{ e_1,\ldots,e_m \}$ be an orthonormal basis of $T_{h_0} \calm$. According to \cite[Lemma 6.1.2]{fillnm} there exists a parametrization $\phi : V \subset \bbr^m \to U \cap \calm$ around $h_0$ such that
\begin{align}\label{para-jumps}
\phi( \langle e,h \rangle ) = h \quad \text{for all $h \in U \cap \calm$,}
\end{align}
where we use the notation $\langle e,h \rangle := (\langle e_1,h \rangle,\ldots,\langle e_m,h \rangle)$. In view of Definition \ref{def-index-set} we may assume, without loss of generality, that $[0,\epsilon] \subset {\rm supp}(F^k)$ for some $\epsilon > 0$. By (\ref{jumps-tangential}), and since $U$ is an open neighborhood of $h_0$, we obtain, after reducing $\epsilon > 0$ if necessary, that
\begin{align}\label{jumps-tangential-2}
h_0 + t \gamma^k(h_0) \in U \cap \calm \quad \text{for all $t \in [0,\epsilon]$.}
\end{align}
Setting $y_0 := \langle e,h_0 \rangle$, by taking into account (\ref{jumps-tangential-2}) and (\ref{para-jumps}) we get
\begin{align*}
\gamma^k(h_0) &= \frac{\partial}{\partial x_k} (h_0 + x_k \gamma^k(h_0)) |_{x=0} = \lim_{t \to 0} \frac{h_0 + t \gamma^k(h_0) - h_0}{t}
\\ &= \lim_{t \to 0} \frac{\phi(y_0 + t \langle e,\gamma^k(h_0) \rangle) - \phi(y_0)}{t} = D \phi(y_0) \langle e,\gamma^k(h_0) \rangle \in T_{h_0} \calm,
\end{align*}
finishing the proof.
\end{proof}

Now, we are ready to present our main results concerning the flatness of invariant manifolds.

\begin{theorem}\label{thm-local}
Suppose that the submanifold $\calm$ is invariant for (\ref{SPDE}). Suppose there exists $d \in \bbn_0$ such that for each $h_0 \in \calm$ we have
\begin{align}\label{cond-local}
d \leq \dim \bigcap_{h \in U \cap \calm} \langle \gamma^k(h) : k \in K \rangle
\end{align}
for some open neighborhood $U \subset H$ of $h_0$. 
\begin{enumerate}
\item Then, for each $h_0 \in \calm$ the following statements are true:
\begin{enumerate}
\item We have $\fl \calm(h_0) \geq d$.

\item There exist an open neighborhood $U_0 \subset H$ of $h_0$, a $d$-dimensional subspace $\call \subset H$ and a finite dimensional $C^3$-submanifold $\caln$ of $\call^{\perp}$ with $\dim \caln = \dim \calm - d$ such that $U_0 \cap \calm = U_0 \cap (\caln \oplus \call)$.

\item If $d = \dim \calm$, then $\calm$ is a local affine space generated by $\call$ around $h_0$.

\item If $d = \dim \calm - 1$, then $\calm$ is a local foliation generated by $\call$ around $h_0$.
\end{enumerate}
\item If, furthermore, the submanifold $\calm$ is connected as a topological subspace of $H$, and we have $\fl \calm(h_0) = d$ for each $h_0 \in \calm$, then the following statements are true:
\begin{enumerate}
\item We have $\fl \calm = d$.

\item There exist a $d$-dimensional subspace $\call \subset H$ and a finite dimensional $C^3$-submanifold $\caln$ of $\call^{\perp}$ with $\dim \caln = \dim \calm - d$ such that $\calm = \caln \oplus \call$.

\item If $d = \dim \calm$, then $\calm$ is an affine space generated by $\call$.

\item If $d = \dim \calm - 1$, then $\calm$ is a foliation generated by $\call$.
\end{enumerate}
\end{enumerate}
\end{theorem}

\begin{proof}
Let $h_0 \in \calm$ be arbitrary. By assumption, there exists a $d$-dimensional subspace $\call_{h_0}$ such that
\begin{align*}
\call_{h_0} \subset \bigcap_{h \in U \cap \calm} \langle \gamma^k(h) : k \in K \rangle,
\end{align*}
and hence, by Proposition \ref{prop-jumps} we obtain
\begin{align*}
\call_{h_0} \subset T_h \calm \quad \text{for all $h \in U \cap \calm$.}
\end{align*}
Therefore, Proposition \ref{prop-m-l-loc} proves the first statement, and the second statement follows from Proposition \ref{prop-connected}.
\end{proof}

Theorem \ref{thm-local} shows that under condition (\ref{cond-local}) on the volatilities $(\gamma^k)_{k \in K}$ invariance of the submanifold implies the inequality $\fl \calm(h_0) \geq d$ concerning its flatness. Roughly speaking, this means that the flatness of the submanifold is at least equal to the number of driving sources with small jumps. Furthermore, the submanifold admits locally a direct sum decomposition into another manifold and a $d$-dimensional linear space. If the submanifold $\calm$ is connected and we even have equality in $\fl \calm(h_0) \geq d$, then the direct sum decomposition holds globally. The following Theorem \ref{thm-global} presents another condition, namely (\ref{cond-global}), on the volatilities $(\gamma^k)_{k \in K}$ under which such a global direct sum decomposition of the manifold holds true.

\begin{theorem}\label{thm-global}
Suppose that the submanifold $\calm$ is invariant for (\ref{SPDE}), and let $d \in \bbn_0$ be such that
\begin{align}\label{cond-global}
d \leq \dim \bigcap_{h \in \calm} \langle \gamma^k(h) : k \in K \rangle.
\end{align}
Then the following statements are true:
\begin{enumerate}
\item We have $\fl \calm \geq d$.

\item There exist a $d$-dimensional subspace $\call \subset H$ and a finite dimensional $C^3$-submanifold $\caln$ of $\call^{\perp}$ with $\dim \caln = \dim \calm - d$ such that $\calm = \caln \oplus \call$.

\item If $d = \dim \calm$, then $\calm$ is an affine space generated by $\call$.

\item If $d = \dim \calm - 1$, then $\calm$ is a foliation generated by $\call$.
\end{enumerate}
\end{theorem}

\begin{proof}
By assumption, there exists a $d$-dimensional subspace $\call$ such that
\begin{align*}
\call \subset \bigcap_{h \in \calm} \langle \gamma^k(h) : k \in K \rangle,
\end{align*}
and hence, by Proposition \ref{prop-jumps} we obtain
\begin{align*}
\call \subset T_h \calm \quad \text{for all $h \in \calm$.}
\end{align*}
Therefore, Proposition \ref{prop-m-l} concludes the proof.
\end{proof}

\section{An example: The L\'{e}vy driven Hull-White extension of the Vasi\u{c}ek model}\label{sec-example}

For the sake of illustration of our previous results, we present an example from mathematical finance, which concerns the modeling of interest rate curves, namely the L\'{e}vy driven Hull-White extension of the Vasi\u{c}ek model, which is an example of the so-called HJMM (Heath-Jarrow-Morton-Musiela) equation
\begin{align}\label{HJMM}
\left\{
\begin{array}{rcl}
d r_t & = & (\frac{d}{d\xi} r_t + \alpha_{\rm HJM}(r_t)) dt + \gamma(r_{t-}) dX_t \medskip
\\ r_0 & = & h_0.
\end{array}
\right.
\end{align}
Here the state space is a suitable Hilbert space $H$ consisting of functions $h : \bbr_+ \to \bbr$ (see, for example, \cite[Sec. 5]{fillnm}), and $\frac{d}{d\xi}$ is the differential operator, which is generated by the translation semigroup on $H$. We refer, e.g., to \cite{Filipovic-Tappe, Barski, P-Z-paper, Marinelli} for the L\'{e}vy driven HJMM equation. In this section, we assume that the L\'{e}vy process is one-dimensional and has the canonical representation $X = W + x * (\mu^X - \nu)$ with a standard Wiener process $W$ such that for some $\epsilon > 0$ we have $[0,\epsilon] \subset {\rm supp}(F)$ or $[-\epsilon,0] \subset {\rm supp}(F)$, where $F$ denotes the L\'{e}vy measure of $X$. For the Hull-White extension of the Vasi\u{c}ek model the volatility $\gamma : H \to H$ is constant, that is $\gamma(h_1) = \gamma(h_2)$ for all $h_1,h_2 \in H$. Therefore, and since $H$ consists of functions mapping $\bbr_+$ to $\bbr$, we agree to write $\gamma(\xi)$ instead of $(\gamma(h))(\xi)$ for $\xi \in \bbr_+$. With this convention, the volatility $\gamma \in H$ is given by
\begin{align*}
\gamma(\xi) = \rho \cdot \exp(-c \xi), \quad \xi \in \bbr_+
\end{align*}
with constants $\rho \neq 0$ and $c \in \bbr$. The drift $\alpha_{\rm HJM} \in H$ is constant as well, and it is given by the HJM drift condition
\begin{align*}
\alpha_{\rm HJM} = - \gamma \cdot \Psi' \bigg( - \int_0^{\bullet} \gamma(\xi) d\xi \bigg),
\end{align*}
where $\Psi$ denotes the cumulant generating function of the L\'{e}vy process $X$. Now, let $\calm$ be a two-dimensional submanifold, which is invariant for (\ref{HJMM}). Then, according to Theorem \ref{thm-global} the submanifold $\calm$ is a foliation generated by $\call = \langle \xi \mapsto \exp(-c \xi) \rangle$. Consequently, for the L\'{e}vy driven Hull-White extension of the Vasi\u{c}ek model with small jumps, every invariant manifold must necessarily be a foliation. It is well-known that, conversely, the Hull-White extension of the Vasi\u{c}ek model admits a two-dimensional realization, that is, for every $h_0 \in \cald(d/d\xi)$ there exists a two-dimensional invariant manifold with $h_0 \in \calm$, where the invariant manifolds are foliations generated by $\call$. For the L\'{e}vy driven case, we refer, for example, to \cite{Tappe-Levy}.

\section*{Acknowledgement}

I am grateful to an anonymous referee for his/her valuable comments and suggestions.

\begin{appendix}

\section{Finite dimensional submanifolds in Hilbert spaces}\label{app-manifolds}

In this appendix, we provide the required results about finite dimensional submanifolds in Hilbert spaces. Let $H$ be a Hilbert space and let $k,m \in \bbn$ be positive integers.

\begin{definition}\label{def-submanifold}
A nonempty subset $\calm \subset H$ is a \emph{$m$-dimensional
$C^k$-submanifold of $H$}, if for all $h_0 \in
\mathcal{M}$ there exist an open neighborhood $U \subset H$ of $h_0$, an open subset $V \subset \mathbb{R}^m$ and a map $\phi \in C^k(V;H)$ such
that
\begin{enumerate}
\item $\phi : V \rightarrow U \cap \mathcal{M}$ is a homeomorphism;
\item $D \phi(y)$ is one to one for all $y \in V$.
\end{enumerate}
The map $\phi$ is called a \emph{parametrization} of $\mathcal{M}$
around $h_0$.
\end{definition}

For what follows, let $\calm$ be a $m$-dimensional $C^k$-submanifold of $H$. For the purpose of this paper, we require the notion of the flatness of $\calm$, which is defined as follows.

\begin{definition}
For $h_0 \in \calm$ we define the \emph{flatness} of $\calm$ at $h_0$, denoted by $\fl \calm(h_0)$, as the largest integer $d \in \{ 0,\ldots,m \}$ such that for some $d$-dimensional subspace $\call \subset H$ and some open neighborhood $U$ of $h_0$ we have
\begin{align*}
\call \subset T_h \calm \quad \text{for all $h \in U \cap \calm$.}
\end{align*}
\end{definition}

\begin{definition}
We call $\fl \calm := \min_{h \in \calm} \fl \calm(h)$ the \emph{flatness} of $\calm$.
\end{definition}

\begin{remark}
A similar notion, which also measures the amount of flatness of a manifold, is the rank, which is defined for complete Riemannian manifolds. We refer, for example, to \cite{BBE}, \cite{Ballmann} or \cite{Spatzier} for the precise definition.  
\end{remark}

\begin{definition}
Let $\call \subset H$ be a finite dimensional subspace.
\begin{enumerate}
\item $\calm$ is an \emph{affine space} generated by $\call$ if there exists an element $g_0 \in \call^{\perp}$ such that $\calm = g_0 \oplus \call$.

\item $\calm$ is a \emph{foliation} generated by $\call$ if there exists a one-dimensional $C^k$-submanifold $\caln$ of $\call^{\perp}$ such that $\calm = \caln \oplus \call$.
\end{enumerate}
\end{definition}

\begin{definition}
Let $\call \subset H$ be a finite dimensional subspace, and let $h_0 \in \calm$ be arbitrary.
\begin{enumerate}
\item $\calm$ is a \emph{local affine space} generated by $\call$ around $h_0$ if there exist an open neighborhood $U$ of $h_0$ and an element $g_0 \in \call^{\perp}$ such that $U \cap \calm = U \cap (g_0 \oplus \call)$.

\item $\calm$ is a \emph{local foliation} generated by $\call$ around $h_0$ if there exist an open neighborhood $U$ of $h_0$ and a one-dimensional $C^k$-submanifold $\caln$ of $\call^{\perp}$ such that $U \cap \calm = U \cap (\caln \oplus \call)$.
\end{enumerate}
\end{definition}

\begin{proposition}\label{prop-m-l-loc}
Let $h_0 \in \calm$ be arbitrary, let $\call \subset H$ be a subspace and let $U \subset H$ be an open neighborhood of $h_0$ such that 
\begin{align}\label{L-tang}
\call \subset T_h \calm \quad \text{for all $h \in U \cap \calm$.} 
\end{align}
Then, denoting by $h_0 = h_1 + h_2$ the direct sum decomposition of $h_0$ according to $H = \call^{\perp} \oplus \call$, there exist open neighborhoods $U_1 \subset \call^{\perp}$ of $h_1$ and $U_2 \subset \call$ of $h_2$ such that $U_0 := U_1 \oplus U_2$ is an open neighborhood of $h_0$ satisfying the following conditions:
\begin{enumerate}
\item We have $U_0 \cap \calm = U_0 \cap ((U_0 \cap \calm) + \call)$.

\item The subset $\caln := U_1 \cap \Pi_{\call^{\perp}} \calm$ is a $C^k$-submanifold of $\call^{\perp}$ with $\dim \caln = \dim \calm - \dim \call$, and we have $U_0 \cap \calm = U_0 \cap (\caln \oplus \call)$.
\end{enumerate}
\end{proposition}

\begin{proof}
Setting $p := \dim \call$, there exists an orthonormal basis $\{ e_1,\ldots,e_m \}$ of $T_{h_0} \calm$ such that $\{ e_1,\ldots,e_p \}$ is an orthonormal basis of $\call$. According to \cite[Lemma 6.1.2]{fillnm} there exists a parametrization $\phi : V' \subset \bbr^m \to U' \cap \calm$ around $h_0$ with $U' \subset U$ such that
\begin{align}\label{para-Damir}
\phi( \langle e,h \rangle ) = h \quad \text{for all $h \in U' \cap \calm$,}
\end{align}
where we use notation $\langle e,h \rangle := (\langle e_1,h \rangle,\ldots,\langle e_m,h \rangle) \in \bbr^m$. Since $U' \subset H$ is an open neighborhood of $h_0$, there exist open neighborhoods $U_{1}' \subset \call^{\perp}$ of $h_{1}$ and $U_{2}' \subset \call$ of $h_{2}$ such that $U_1' \oplus U_2' \subset U'$. By (\ref{para-Damir}) we have
\begin{align}\label{prod-mani}
\phi^{-1}(U_1' \cap \calm) \subset \bbr^{m-p} \quad \text{and} \quad \phi^{-1}(U_2' \cap \calm) \subset \bbr^p 
\end{align}
with respect to the direct sum decomposition $\bbr^m = \bbr^{m-p} \oplus \bbr^{p}$. Since $V'$ is open in $\bbr^m$, there are open subsets $V_1 \subset \bbr^{m-p}$ and $V_2 \subset \bbr^{p}$ such that $V_0 \subset V'$, where $V_0 := V_1 \oplus V_2$. Since $\phi$ is a homeomorphism, there exists an open neighborhood $U_0'$ of $h_0$ such that $\phi(V_0) = U_0' \cap \calm$. By (\ref{prod-mani}) there exist open neighborhoods $U_1 \subset \call^{\perp}$ of $h_1$ and $\widetilde{U}_2 \subset \call$ of $h_2$ such that $(U_1 \oplus \widetilde{U}_2) \cap \calm = U_0' \cap \calm$. Setting $\caln := U_1 \cap \Pi_{\call^{\perp}} \calm$, $U_2 := \widetilde{U}_2 \cap \Pi_{\call} \calm$ and $U_0 := U_1 \oplus U_2$, we have $\Pi_{\call} U_0 = U_2$ and
\begin{align*}
\phi(V_0) = U_0 \cap \calm = \caln \oplus U_2,
\end{align*}
and it follows that
\begin{align}\label{converse-impl}
U_0 \cap \calm = U_0 \cap (\caln \oplus U_2) \subset U_0 \cap (\caln \oplus \call).
\end{align}
Defining the mappings $\phi_1 := \phi|_{V_1}$ and $\phi_2 := \phi|_{V_2}$, we obtain:
\begin{itemize}
\item $\phi_1 \in C^k(V_1;\call^{\perp})$ and $\phi_2 \in C^k(V_2;\call)$, because $\phi \in C^k(V_0;H)$.

\item $\phi_1 : V_1 \to \caln$ and $\phi_2 : V_2 \to U_2$ are homeomorphisms, because $\phi : V_0 \to \caln \oplus U_2$ is a homeomorphism.

\item For all $y_1 \in V_1$ and $y_2 \in V_2$ the mappings $D \phi_1(y_1)$ and $D \phi_2(y_2)$ are one to one, because
\begin{align*}
D \phi(y_1 + y_2) = D \phi_1(y_1) + D \phi_2(y_2)
\end{align*}
is one to one.
\end{itemize}
Therefore, $\caln$ is a $(m-p)$-dimensional submanifold of $\call^{\perp}$ with parametrization $\phi_1$, and $U_2$ is a $p$-dimensional submanifold of $\call$ with parametrization $\phi_2$. Furthermore, by (\ref{para-Damir}) there is an isomorphism $T : \bbr^p \to \call$ such that $\phi_2 = T|_{V_2}$, and hence, we have
\begin{align*}
\phi(y_1 + y_2) = \phi_1(y_1) + T y_2 \quad \text{for all $y_1 \in V_1$ and $y_2 \in V_2$.}
\end{align*}
Now, we will show that
\begin{align}\label{inclusion}
U_0 \cap ( (U_0 \cap \calm) + \call ) \subset U_0 \cap \calm.
\end{align}
Indeed, let $h \in U_0 \cap \calm$ and $g \in \call$ be such that $h + g \in U_0$. Then there exist unique $y_1 \in V_1$, $y_2 \in V_2$ and $z_2 \in \bbr^p$ such that $h = \phi_1(y_1) + T y_2$ and $g = Tz_2$, and we obtain
\begin{align*}
h + g = \phi_1(y_1) + T (y_2 + z_2). 
\end{align*}
Since $h + g \in U_0$ and $\Pi_{\call} U_0 = U_2$, we have $T(y_2 + z_2) \in U_2$. Therefore, and since $T : \bbr^p \to \call$ is an isomorphism, we obtain $y_2+z_2 \in V_2$, and hence
\begin{align*}
h + g = \phi(y_1+(y_2+z_2)) \in U_0 \cap \calm,
\end{align*}
proving (\ref{inclusion}). In order to prove the converse inclusion of (\ref{converse-impl}), let $h \in \caln$ and $g \in \call$ be such that $h+g \in U_0$. There exists $f \in \call$ such that $h + f \in U_0 \cap \calm$. Thus, we have $h+g = (h+f) + (g-f) \in U_0 \cap \calm + \call$. Since $h+g \in U_0$, by (\ref{inclusion}) we obtain $h+g \in U_0 \cap \calm$, completing the proof.
\end{proof}

\begin{proposition}\label{prop-m-l}
Suppose that $\calm$ is closed as a subset of $H$, and let $\call \subset H$ be a subspace such that 
\begin{align}\label{call-tang-global}
\call \subset T_h \calm \quad \text{for all $h \in \calm$.} 
\end{align}
Then the following statements are true:
\begin{enumerate}
\item We have $\calm = \calm + \call$.

\item The subset $\caln := \Pi_{\call^{\perp}} \calm$ is a $C^k$-submanifold of $\call^{\perp}$ with $\dim \caln = \dim \calm - \dim \call$, and we have $\calm = \caln \oplus \call$.
\end{enumerate}
\end{proposition}

\begin{proof}
In order to prove $\calm + \call \subset \calm$, let $h \in \calm$ and $g \in \call$ be arbitrary, and suppose that $h + g \notin \calm$. We define $t \in [0,1]$ as
\begin{align*}
t := \inf \{ s \in [0,1] : h + sg \notin \calm \},
\end{align*}
and set $h_0 := h + t g$. Since $\calm$ is closed as a subset of $H$, we have $h_0 \in \calm$, which implies $t < 1$. Furthermore, there exists a sequence $(s_n)_{n \in \bbn} \subset (0,\infty)$ with $s_n \to 0$ such that $h_0 + s_n g \notin \calm$ for all $n \in \bbn$. By Proposition \ref{prop-m-l-loc} there exists an open neighborhood $U$ of $h_0$ such that
\begin{align*}
U \cap \calm = U \cap ((U \cap \calm) + \call),
\end{align*}
which contradicts $h_0 + s_n g \notin \calm$ for all $n \in \bbn$, establishing the first statement.

According to Proposition \ref{prop-m-l-loc}, the subset $\caln := \Pi_{\call^{\perp}} \calm$ is a $C^k$-submanifold of $\call^{\perp}$ with $\dim \caln = \dim \calm - \dim \call$. Furthermore, we have $\calm \subset \caln \oplus \call$. In order to prove the converse inclusion $\caln \oplus \call \subset \calm$, let $h \in \caln$ and $g \in \call$ be arbitrary. There exists $f \in \call$ such that $h + f \in \calm$. Thus, we have $h+g = (h+f) + (g-f) \in \calm + \call$, and we obtain $\caln \oplus \call \subset \calm + \call = \calm$, establishing the second statement.
\end{proof}

\begin{proposition}\label{prop-connected}
Suppose that the submanifold $\calm$ is connected as a topological subspace of $H$, and let $d \in \bbn_0$ be such that $\fl \calm(h_0) = d$ for each $h_0 \in \calm$. Then there exist a subspace $\call \subset H$ with $\dim \call = d$ and a finite dimensional $C^k$-submanifold $\caln$ of $\call^{\perp}$ with $\dim \caln = m - d$ such that $\calm = \caln \oplus \call$.
\end{proposition}

\begin{proof}
For each $h_0 \in \calm$ there exist a $d$-dimensional subspace $\call_{h_0} \subset H$ and an open neighborhood $U_{h_0}$ of $h_0$ such that
\begin{align}\label{tangent-connected}
\call_{h_0} \subset T_h \calm \quad \text{for all $h \in U_{h_0} \cap \calm$.}
\end{align}
We will show that
\begin{align}\label{call-g-h}
\call_{g_0} = \call_{h_0} \quad \text{for all $g_0,h_0 \in \calm$.}
\end{align}
Let $g_0,h_0 \in \calm$ be arbitrary. Since the submanifold $\calm$ is locally path-connected and connected, it is even path-connected, see, for example, \cite[Prop. 1.6.7]{Abraham}. Thus, there exists a continuous function $f : I \to \calm$ with $f(0) = g_0$ and $f(1) = h_0$, where $I = [0,1]$. Since the graph $f(I) \subset \calm$ is compact, there exist an integer $n \in \bbn$ and elements $g_1,\ldots,g_n \in f(I)$ with $g_n = h_0$ such that
\begin{align*}
f(I) = f(I) \cap \bigg( \bigcup_{k=0}^n U_{g_k} \bigg).
\end{align*}
We define an integer $e \in \{ 1,\ldots,n \}$, elements $0 = t_0 < \ldots < t_e \leq 1$ and pairwise different $\pi(0),\ldots,\pi(e) \in \{ 0,\ldots,n \}$ with $\pi(0) = 0$, $\pi(e) = n$ and $f(t_k) \in \bigcup_{i=0}^k U_{g_{\pi(i)}}$, $f(t_k) \notin \bigcup_{i=0}^{k-1} U_{g_{\pi(i)}}$ for $k=0,\ldots,e$ inductively as follows:
\begin{itemize}
\item We set $t_0 := 0$ and $\pi(0) := 0$.

\item For the induction step $k \to k+1$ let $k \in \{ 0,\ldots,n-1 \}$ be arbitrary.
\begin{itemize}
\item If $f(t_k) \in U_{h_0}$, then we set $e := k$.

\item Otherwise, we define $t_{k+1} \in [t_k,1]$ as
\begin{align}\label{def-tk}
t_{k+1} := \inf \bigg\{ t \in [t_k,1] : f(t) \notin \bigcup_{i=0}^k U_{g_{\pi(i)}} \bigg\}.
\end{align}
By the continuity of $f$ we have 
\begin{align*}
t_{k+1} > t_k \quad \text{and} \quad f(t_{k+1}) \notin \bigcup_{i=0}^k U_{g_{\pi(i)}}. 
\end{align*}
Thus, there exists an index $l \in \{ 1,\ldots,n \}$ with $l \notin \{ \pi(1),\ldots,\pi(k) \}$ such that $f(t_{k+1}) \in U_{g_l}$. We set $\pi(k+1) := l$.
\end{itemize}
\end{itemize}
Now, by induction we prove that
\begin{align}\label{spaces-equal}
\call_{g_{\pi(0)}} = \call_{g_{\pi(k)}} \quad \text{for all $k=0,\ldots,e$.}
\end{align}
For the induction step $k \to k+1$, by the definition (\ref{def-tk}) of $t_{k+1}$ we have
\begin{align*}
f(s) \in \bigcup_{i=0}^k U_{g_{\pi(i)}} \quad \text{for all $s \in [t_k, t_{k+1})$.}
\end{align*}
Moreover, by the continuity of $f$ there exists $\delta > 0$ with $t_k < t_{k+1} - \delta$ such that
\begin{align*}
f(s) \in U_{g_{\pi(k+1)}} \quad \text{for all $s \in (t_{k+1} - \delta, t_{k+1}]$.}
\end{align*}
Therefore, we obtain
\begin{align*}
f(s) \in \bigcup_{i=0}^k \big( U_{g_{\pi(i)}} \cap U_{g_{\pi(k+1)}} \big) \quad \text{for all $s \in (t_{k+1} - \delta, t_{k+1})$.}
\end{align*}
Hence, there exist $i \in \{ 0,\ldots,k \}$ and $s \in (t_{k+1} - \delta, t_{k+1})$ such that $U := U_{g_{\pi(i)}} \cap U_{g_{\pi(k+1)}}$ is an open neighborhood of $f(s)$. By (\ref{tangent-connected}) we obtain
\begin{align*}
\call_{g_{\pi(i)}} + \call_{g_{\pi(k+1)}} \subset T_h \calm \quad \text{for all $h \in U \cap \calm$.}
\end{align*}
Since $\fl \calm(f(s)) = d$, we deduce that $\call_{g_{\pi(i)}} = \call_{g_{\pi(k+1)}}$, which completes the induction step, and establishes (\ref{spaces-equal}), whence we arrive at (\ref{call-g-h}). Therefore, and by (\ref{tangent-connected}) there exists a $d$-dimensional subspace $\call$ such that (\ref{call-tang-global}) is fulfilled. Consequently, applying Proposition \ref{prop-m-l} finishes the proof.
\end{proof}

\end{appendix}

\end{document}